\documentclass{amsart}
\usepackage{amsmath}
\usepackage{amssymb}
\usepackage[all]{xy}
\input xy
\xyoption{all} \CompileMatrices

\newtheorem{Th}{\bf Theorem}[section]
\newtheorem{Prop}[Th]{\bf Proposition}
\newtheorem{Le}[Th]{\bf Lemma}

\newtheorem{Df}[Th]{\bf Definition}
\newtheorem{remark}[Th]{\bf Remark}

\newcommand{\NN}{\mathbb{N}}

\newcommand{\CC}{\mathbb{C}}

\newcommand{\JJ}{\mathbb{J}}
\newcommand{\EE}{\mathbb{E}}

\begin{document}

\title{\bfseries ON PROPER AND EXTERIOR SEQUENTIALITY}

\author{L. Espa\~{n}ol, J.M. Garc\'{\i}a-Calcines and M.C. M\'{\i}nguez.}



\address{L. Espa\~{n}ol Gonz\'{a}lez; M.C. M\'{\i}nguez Herrero \newline
\indent Departamento de Matem\'{a}ticas y Computaci\'{o}n \newline
\indent Universidad de La Rioja \newline \indent 26004 LOGRO\~{N}O
} \email{luis.espanol@dmc.unirioja.es}
\email{carmen.minguez@dmc.unirioja.es}

\address{J.M. Garc\'{\i}a Calcines \newline
\indent Departamento de Matem\'{a}tica Fundamental \newline
\indent Universidad de La Laguna \newline \indent 38271 LA
LAGUNA.} \email{jmgarcal@ull.es}

\thanks{Supported by MEC projects (Spain) BFM2000-0961, MTM2006-06317,
the UR project ACPI2001/10 PR I+D and FEDER }

\date{\today}

\begin{abstract}
In this article a sequential theory in the category of spaces and
proper maps is described and developed. As a natural extension a
sequential theory for exterior spaces and maps is obtained.
\end{abstract}

\maketitle

\section{Introduction}

Many mathematicians have been interested in studying relationships
between the topology of a space and its convergent sequences.
Among these we can mention M. Fr\'{e}chet \cite{Fre}, P. Urysohn
\cite{U}, A. Arhangel'skii \cite{A} or  J. Kisynski \cite{K}. In
the sixties, R. M. Dudley \cite{Dud} suggested to re-examine
topology from a sequential point of view; and S. P. Franklin in
1965 \cite{Fr} arrived at the satisfactory notion of sequential
space. Sequential spaces are the most general class of spaces for
which sequences suffice to determine the topology. This class of
spaces is so large that it includes the most important and useful
examples of topological spaces, such as CW-complexes, metric
spaces or topological manifolds. They also form a coreflective
subcategory $\mathbf{Seq}$ \cite{Ba} of the category
$\mathbf{Top}$ of topological spaces, and have good categorical
properties, including being complete, cocomplete and cartesian
closed. Sequential space theory not only interacts in general
topology and analysis, but also in topos theory, as it was shown
in \cite{J} by P.T. Johnstone. In his work P.T. Johnstone
presented the category of sequential spaces as a certain
subcategory of a topos of sheaves, where the embedding preserves
some useful colimits and exponentials.

It is natural to ask for a sequential theory in the category of
spaces and proper maps. Continuing Dudley's program, R. Brown
studied in 1973 the sequential versions of proper maps and of
one-point compactification \cite{Br}. In that paper he gives
sufficient conditions on any space $X$ to ensure that its
Alexandroff compactification $X^+$ is sequential. He also asks for
more general conditions for $X^+$ to be sequential. However, until
now, it seems that nobody has noticed the lack of an analogue
notion of sequential space in the proper scope. In this article we
present a proper sequential theory for spaces and solve the
questions posed by R. Brown. In order to do this we firstly use a
slightly different notion of sequentially proper map. This notion
turns out to be equivalent to that given by R. Brown when the
spaces have some natural sequential properties. Then we introduce
what we call {\it $\omega$-sequential spaces}, in which the closed
compact subsets are completely determined by the proper sequences
(and the convergent ones). This new class of spaces plays in the
category of spaces and proper maps $\mathbf{P}$ the same role
played in $\mathbf{Top}$ by the sequential spaces. We also show
that the $\omega$-sequential spaces verify the expected natural
properties. Perhaps one of the most significant properties is that
given any map $f:X\rightarrow Y$ between topological spaces, where
$X$ is $\omega $-sequential, then $f$ is proper if and only if $f$
preserves convergent sequences and proper sequences (see
Proposition \ref{omega2} (ii) for more details). We also check
that the $\omega $-sequential spaces contain enough important
examples, where the CW-complexes, the metric spaces and the
topological manifolds are included.

As it is known, $\mathbf{P}$ does not have good properties as far
as limits and colimits is concerned so many topological
constructions are not possible in the proper category.
Subsequently, several classical sequential results cannot be
transferred to the proper category.  A way to solve this problem
is to consider a greater category having better properties and
then define a convenient notion of sequential object. For
instance, we could use the Edwards-Hastings embedding \cite{E-H}
of the proper category of locally compact $\sigma $-compact
Hausdorff spaces into the category of pro-spaces. From our point
of view this embedding has important disadvantages. On the one
hand, it is necessary to consider strong restrictions of the
proper category. On the other hand many constructions give rise to
pro-spaces that cannot be interpreted as regular spaces. The
category of exterior spaces (see \cite{G-G-H,G-H}) is a good
possibility. Broadly speaking, an exterior space is a topological
space with a `neighborhood system at infinity' which we call
externology, while an exterior map is a continuous map which is
`continuous at infinity'. This category not only contains in its
totality the proper category but also does not lose the geometric
notion. Furthermore, it is complete and cocomplete. Simplicity in
their description and the similarity with the classical limit and
colimit constructions turn the exterior spaces to a useful and
powerful tool for the study of non-compact spaces.

Taking into account the above discussion, the last part of this
paper is devoted to extend our proper sequential theory to the
category $\mathbf{E}$ of exterior spaces. Such extension begins
establishing the definition of {\it e-sequential exterior space}
in such a way that an $\omega $-sequential space is a particular
case. We will analyze the new resulting category verifying that it
has analogous properties to those of the classic sequential
spaces. In particular the category $\mathbf{Eseq}$ of e-sequential
exterior spaces is a coreflective subcategory of $\mathbf{E}$, a
property which is not inherited in the proper case. Finally, and
similarly as done in the classical topological case by P.T.
Johnstone, we also prove that the category $\mathbf{Eseq}$ is a
full subcategory of a topos of sheaves.

\subsection*{Acknowledgements}
We want to acknowledge to I. Gotchev for providing us the useful
characterization of sequentiality compact spaces in terms of open
covers.

\section{Preliminary definitions and properties}

First of all, we will establish the more relevant notions and
properties, as well as their corresponding notation that will be
used throughout this work.

Recall that a proper map is a continuous map $f:X\rightarrow Y$
such that $f^{-1}(K)$ is closed compact, for every closed compact
subset $K\subset Y.$ Taking the filter of open subsets in a space
$(X,\tau _X)$
$$\varepsilon_{cc}(X)=\{U\in\tau_X;U^c\; \text{is compact}\},$$
\noindent where $U^c$ denotes the complement of $U$ in $X$, then
it is easy to check that $f$ is proper if and only if it is
continuous and $f^{-1}(V)\in\varepsilon_{cc}(X)$, $\forall
V\in\varepsilon_{cc}(Y).$

We will denote by $\mathbf{P}$ the category of spaces and proper
maps. If we consider for any space $X$ its Alexandroff
compactification as a based space $(X^+,\infty)$ then we have a
functor $(-)^+:\mathbf{P}\rightarrow \mathbf{Top^{\infty}}$. Here
$\mathbf{Top^{\infty}}$ denotes the subcategory  of
$\mathbf{Top^*}$ (the category of all based spaces and maps) whose
objects are based spaces $(X,x_{0})$ such that $\{x_{0}\}$ is
closed in $X$, and the morphisms are based maps $f:(X,x_{0})\to
(Y,y_{0})$ verifying $f^{-1}(\{y_{0}\})=\{x_{0}\}$. On the set of
natural numbers $\NN$ we consider the discrete topology, so
$\NN^+$ represents the sequence of natural numbers with its limit
$\{\infty \}$. On any space $X$ we have sequences as (continuous)
maps $s:\NN\to X,$ $s(n)=x_n$. The set of all sequences in $X$
will be denoted by $X^{\NN }$. On the other hand a continuous map
$s:\NN^+\to X$ is a sequence $s$ and a limit point $s(\infty)$. A
sequence $s:\NN\to X$ is said to be \emph{convergent} if it
factorizes through a continuous map $\NN^+\to X$. The set of all
convergent sequences of a space $X$ will be denoted by
$\Sigma_c(X)\subset X^\NN$, and we will write
$\Sigma_c^+(X)=\mathbf{Top}(\NN^+,X)$. We are interested in the
monoid $\EE$ of all monotone and injective sequences
$u:\NN\to\NN$, because each composition $s\circ u$ is a
subsequence of $s$, for all $u\in\EE$. Note that $\EE\subset
\mathbf{P}(\NN ,\NN)$. It is clear that for any space $X$,
$\Sigma_c(X)$ is an $\EE$-set with the (right) action given by
composition. The same is true for $X^\NN$, and in order to
`measure the convergency' of an arbitrary sequence $s:\NN\to X$ we
can use the ideal
$$\langle s\in\Sigma_c(X)\rangle=\{u\in\EE;s\circ
u\in\Sigma_c(X)\}$$ of $\EE$, which is equal to $\EE$ if and only
if $s$ is convergent. We can also consider those sequences without
convergent subsequences, that is, sequences $s$ such that $\langle
s\in\Sigma_c(X)\rangle=\emptyset$. We will denote by
$\neg\Sigma_c(X)\subset X^\NN$ the set of all such sequences,
where the symbol $\neg$ is used because it corresponds to the
negation in the Heyting algebra \cite{Jo} of all $\EE$-subsets of
$X^\NN$. In this way $\neg\Sigma_c(X)$ is the biggest $\EE$-subset
$E$ of $X^\NN$ such that $\Sigma_c(X)\cap E=\emptyset$.

Topological notions are defined in terms of open sets, and a
notion is called \emph{sequential} or \emph{sequentially defined}
when it is defined in terms of sequences. There are the typical
\emph{sequential spaces}, those which can be characterized by
their convergent sequences. A different kind of sequentiality is
that of \emph{sequential bornological spaces}, which can be
characterized by their bounded sequences. Now we are interested in
the sequentiality with respect to proper sequences. Proper maps
are continuous maps, so proper sequentiality shall be related to
the sequentiality by convergent sequences. We will assume that the
reader is familiarized to sequential spaces. Nevertheless,
although the reader may also refers to S. P. Franklin's article
\cite{Fr} we will recall here some basic notions.

Given a space $X=(X,\tau)$, a subset $U\subset X$ is said to be
\emph{sequentially open} if any sequence $s:\NN\to X$ that has a
limit point $x\in U$ is eventually in $U$. This can be stated in
the form
$$s\to x\in U \;\;\text{implies}\;\; s\propto U,$$ where $s\propto U$
means that $s^{-1}(U)$ is cofinite. The family of all sequentially
open subsets in $X$ is a topology $\tau_{seq}$. A space $X$ is
said to be \emph{sequential} if the open subsets agree with the
sequentially open subsets, that is, $\tau_{seq}=\tau$. A subset
$C\subset X$ is \emph{sequentially closed} if its complement in
$X$ is sequentially open, that is, $$s\to x,\; s\propto C
\;\;\text{implies}\;\; x\in C.$$  A map $f:X\to Y$ is
\emph{sequentially continuous} if it preserves convergent
sequences and limits, that is, $f\circ s\in \Sigma_c^+(Y)$, for
all $s\in \Sigma_c^+(X);$ in order to express this condition we
will write $f\circ\Sigma_c^+(X)\subset\Sigma_c^+(Y).$ Every
continuous map is sequentially continuous but in general the
converse is not true.  We will denote by $\mathbf{Seq}$ the full
subcategory of $\mathbf{Top}$ defined by the sequential spaces. It
is well known that $\mathbf{Seq}$ is a coreflective subcategory
(see \cite{Ba}), and has much better categorical properties than
$\mathbf{Top}$.

Now recall that a \emph{sequentially Hausdorff} space (or a space
which has unique sequential limits) is a space $X$ such that every
convergent sequence $s$ has a unique limit point, that is, the
factorization of $s$ trough $\NN^+$ is unique if it exists. Most
of the well known spaces are sequentially Hausdorff, but not all
of them. For instance, let $\mathbf{2}_S$ denote the space given
by the set $\mathbf{2}=\{0,1\}$ with the Sierpinski topology
$\tau_{\mathbf{2}_S}=\{\mathbf{2},\emptyset ,\{1\}\}.$ Giving a
continuous map $\varphi:X\to\mathbf{2}_S$ is the same as giving an
open subset $U=\varphi^{-1}(1)$; hence the Sierpinski space
$\mathbf{2}_S$ is not sequentially Hausdorff because there exists
two open subsets of $\NN^+$ containing $\NN$.

\begin{remark}
It is also important to remark that every sequentially Hausdorff
space is $T_1$ (so $T_0$) by \cite[Th.1]{W}. \end{remark} The next
definition will be crucial for our purposes.

\begin{Df}
We say that a space $X$ is $S_2$ (or $X$ is an $S_2$-space) when
it is sequential and sequentially Hausdorff.
\end{Df}

Recall also that a space is \emph{sequentially compact} if any
sequence has a convergent subsequence. Considering the
$\EE$-subsets of $X^\NN$, this means that $\langle
s\in\Sigma_c(X)\rangle\neq\emptyset$ for any  sequence $s$, that
is, $\neg\Sigma_c(X)=\emptyset$. The following theorem, due to I.
Gotchev and H. Minchev \cite{G-M}, characterizes sequential
compactness and will be important for the next result. Here, a
sequentially open cover of $X$ means a cover whose elements are
sequentially open sets.

\begin{Th}\label{Got}
For a $T_0$ topological space $X$ the following conditions are
equivalent:
\begin{enumerate}
\item[(i)] $X$ is a sequentially compact space.

\item[(ii)] Every countable sequentially open cover of $X$ has a
finite subcover.
\end{enumerate}
\end{Th}

A \textit{proper sequence} in $X$ is just a proper map
$s:\mathbb{N}\rightarrow X$, that is, a sequence such that
$s^{-1}(K)$ is finite for every closed compact subset $K\subset
X$; in other words, $s\propto U$ for all
$U\in\varepsilon_{cc}(X)$. Denoting by $\Sigma_p(X)$ the set of
the proper sequences in $X$, the following result shows how
different are proper sequences and convergent sequences in
$S_2$-spaces.

\begin{Th}\label{sucprop}
Let $X$ be an $S_2$-space. Then $\Sigma_p(X)=\neg\Sigma_c(X)$.
\end{Th}

\begin{proof}[Proof]

In order to prove that $\neg\Sigma_c(X)\subset\Sigma_p(X)$,
suppose a non-proper sequence $s:\mathbb{N}\rightarrow X$. Then
there exists a closed compact subset $K\subset X$ such that
$s^{-1}(K)$ is infinite; therefore we can find a subsequence $t$
of $s$ contained in $K$. But $K$ is also sequentially compact by
Theorem \ref{Got}, so there is a convergent subsequence of $t$ in
$K$. Hence $s$ is not in $\neg\Sigma_c(X)$.

Conversely, using the fact that the statement
$\Sigma_p(X)\subset\neg\Sigma_c(X)$ is equivalent to
$\Sigma_p(X)\cap\Sigma_c(X)=\emptyset$, consider a convergent
sequence $s:\mathbb{N}\to X$ with limit $x_0\in X$. Then
$K=Im(s)\cup \{x_0\}$ is compact in the $S_2$-space $X$, so $K$ is
closed. Hence $s^{-1}(K)$ is not finite and $s$ can not be proper.
\end{proof}

\begin{remark}Note that if $X$ is Hausdorff then the set $K$ in the above
proof is also closed, so $\Sigma_p(X)\subset\neg\Sigma_c(X)$ by
the proof of Theorem \ref{sucprop}.\end{remark}

In general spaces, proper sequences in $X$ are convergent
sequences in its Alexandroff compactification $X^+$ with limit the
based point $\infty$. We have the following statement, which is
easy to prove.

\begin{Prop}\label{equi}
The functor $(-)^+:\mathbf{P}\to \mathbf{Top^{\infty}}$ is full
and faithful, and it induces an equivalence
$$\mathbf{P}_{lcH}\simeq \mathbf{Top}^{\infty }_{cH}$$ between the
full subcategory $\mathbf{P}_{lcH}$ of $\mathbf{P}$ whose objects
are locally compact Hausdorff spaces and the full subcategory
$\mathbf{Top}^{\infty }_{cH}$ of $\mathbf{Top}^{\infty }$ whose
objects are compact Hausdorff spaces.
\end{Prop}

The quasi-inverse of $(-)^+$ is obtained as follows: Given a based
space $(X,x_0)$ we take $\bar{X}=X-\{x_0\}$ equipped with the
relative topology $\tau _{\bar{X}}=\{A-\{x_0\}|A\in \tau_X\}$.

The condition of being Hausdorff cannot be removed in Proposition
\ref{equi}. The based space $(\mathbf{2}_S,0)$ is compact but it
does not come from the Alexandroff compactification. Otherwise,
$(\mathbf{1}^+,\infty )\cong (\mathbf{2}_S,0)$ would be
homeomorphic, where $\mathbf{1}=\{0\}$ is the one-point space. But
this is impossible.

\section{On a proper notion of sequentiality: $\omega$-sequential spaces}

In this section we will give the notion of $\omega$-sequential
space, which is the core of this work. These spaces play in the
proper case an analogous role to that played by the sequential
spaces since proper maps and sequentially proper maps between them
agree (see (ii) in Proposition \ref{omega2}). We will also see
that the class of $\omega $-sequential spaces contains, among
others, the CW-complexes, the metric spaces and the topological
manifolds.

Before going to our definition we will firstly deal with the
proper sequentiality of maps.

\subsection{Sequentially proper maps}

Since a proper map is a continuous map with a condition on
(closed) compact subsets, a sequentially proper map will be a
sequentially continuous map with a condition on proper sequences.

\begin{Df}\label{seprop}
Given spaces $X,Y$, a map $f:X\to Y$ is \emph{sequentially proper}
if it is sequentially continuous and it preserves proper
sequences, that is,
\begin{enumerate}
\item
$f\circ s\in \Sigma_c^+(Y),$ for all $s\in \Sigma_c^+(X);$ and

\item
$f\circ s\in \Sigma_p(Y),$ for all $s\in \Sigma_p(X).$
\end{enumerate}
In other words, $f\circ\Sigma_c^+(X)\subset\Sigma_c^+(Y)
\;\;\text{and}\;\; f\circ\Sigma_p(X)\subset\Sigma_p(Y).$
\end{Df}

There is an almost obvious relationship between the sequentially
proper maps and the sequentially continuous maps when we consider
Alexandroff compactifications.

\begin{Th}\label{equiva-prop}
Let $f:X\to Y$ be a map between spaces. Then $f$ is sequentially
proper if and only if $f^+:X^+\to Y^+$ is sequentially continuous.
\end{Th}

\begin{proof}[Proof]
The fact that $f:X\to Y$ is sequentially continuous means that
$f^+:X^+\to Y^+$ preserves convergent sequences with limit in $X$.
Similarly, the fact that $f:X\to Y$ preserves proper sequences
means that $f^+:X^+\to Y^+$ preserves convergent sequences with
$\infty$ as limit, so the theorem follows. We leave the details to
the reader.
\end{proof}

Now we compare our Definition \ref{seprop} with a similar notion
given by R. Brown \cite{Br}. We say that $f:X\to Y$ is a
\textit{sequentially proper map in the sense of Brown} if it is
sequentially continuous and $f\times 1_Z:X\times Z\rightarrow
Y\times Z$ is sequentially closed (that is, preserves sequentially
closed subsets), for every space $Z$. In order to avoid confusion
with our definition we will say that a sequentially proper map in
the sense of Brown is a \emph{B-proper map}. The following theorem
relates B-proper maps and sequentially proper maps.

\begin{Th}Let $f:X\rightarrow Y$ be a function between $S_2$-spaces.
Then $f$ is sequentially proper if and only if $f$ is B-proper.
\end{Th}

\begin{proof}[Proof]
By Theorem \ref{sucprop} the following conditions are equivalent
for $S_2$-spaces:

(i) $f\circ\Sigma_p(X)\subset\Sigma_p(Y)$

(ii) $f\circ\neg\Sigma_c(X)\subset\neg\Sigma_c(Y)$

Suppose that $f$ is sequentially continuous. Then (i) means that
$f$ is sequentially proper, and (ii) means that $f$ is B-proper by
\cite[Th. 2.6]{Br}.
\end{proof}

\begin{remark}
Note that the statement `$f$ B-proper implies $f$ sequentially
proper' is true when $X$ is just Hausdorff instead of being $S_2$,
and the converse is also true when $Y$ is just Hausdorff instead
of being $S_2$.\end{remark}

It is clear that every proper map is sequentially proper. The
following task is to find the class of spaces $X,Y$ in which the
proper maps $f:X\to Y$ and the sequentially proper maps agree.

\subsection{s-compact subsets and $\omega$-sequential spaces}

Now we give a sequential notion which is weaker than that of
closed compact subset and well adapted to proper sequentiality.
When these two families of subsets agree we have a
$\omega$-sequential space.

\begin{Df}\label{omega}
Let $X$ be a space.

(i) We say that $C\subset X$ is \emph{s-compact} if $C$ is
sequentially closed and every proper sequence is eventually in the
complement of $C$.

(ii) $X$ is said to be $\omega$\emph{-sequential} if it is a
sequential space and the s-compact subsets agree with the closed
compact subsets.
\end{Df}

For any space $X$ we introduce the family of sequentially open
subsets
$$\varepsilon_{sc}(X)=\{U\subset X;U^c\; \text{is s-compact}\}.$$
Note that in an $\omega$-sequential space $X$, the family
$\varepsilon_{sc}(X)$ is a filter of open subsets. The latter
notation was given to express the following immediate result. The
proof is straightforward and left to the reader.

\begin{Prop}\label{omega2} Let $f:X\rightarrow Y$ be a map between topological
spaces. Then
\begin{enumerate}
\item[(i)] $f:X\rightarrow Y$ is a sequentially proper map if
and only if it is sequentially continuous and
$f^{-1}(V)\in\varepsilon_{sc}(X),$ for all
$V\in\varepsilon_{sc}(Y).$

\item[(ii)] Suppose that $X$ is $\omega$-sequential. Then $f$ is
proper if and only if $f$ is sequentially proper.
\end{enumerate}
\end{Prop}

Now we will give some interesting properties about
$\omega$-sequential spaces. In order to do this we must give
relationships between s-compact, sequentially compact and
countably compact subsets. Under certain weak properties on the
space $X$ these subsets agree. Recall that a space is said to be
\textit{countably compact} when every countable open cover has a
finite subcover.

\begin{Le}\label{lemon}
Let $X$ be a space and $C\subset X$. If $C$ is closed and
s-compact then $C$ is countably compact. In particular, when $X$
is a sequential space, every s-compact subset is countably
compact.
\end{Le}

\begin{proof}[Proof]

Suppose that an open cover $C=\bigcup _{n\in \mathbb{N}}U_n$ does
not admit any finite subcover. Then for all $k\in \mathbb{N}$
$$W_k=\bigcup _{i=1}^kU_i\varsubsetneq C,$$ and the sequence of
open subsets $W_1\subset W_2\subset\cdots\subset W_n\subset\cdots$
is such that for any $n$ there exists $k>n$ such that
$W_n\varsubsetneq W_k$. Set $n_1=1$ and pick a point $x_1\in W_1$;
next consider the smallest natural $n_2>n_1$ such that
$W_{n_1}\varsubsetneq W_{n_2}$ and pick $x_2\in W_{n_2}-W_{n_1}$.
Thus, coming from an inductive process, we obtain a strictly
increasing sequence of natural numbers $u\in\EE$, defined as
$u(k)=n_k,$ and a sequence $s:\NN\to X, s(k)=x_k$, such that
$s(1)\in W_1, s(k+1)\in W_{u(k+1)}-W_{u(k)}, k\in \NN$. This
sequence is proper; indeed, if $L$ is any closed compact subset of
$X$ then $K=L\cap C$ is a closed compact verifying that
$s^{-1}(K)=s^{-1}(L)$. But $K\subset C=\bigcup _{k=1}^{\infty
}W_k$ implies $K\subset \bigcup_{i=1}^mW_{k_i}=W_p$ for some $p$,
so $s^{-1}(K)\subset s^{-1}(W_p)=\{n_1,n_2,...,n_p\}$ is finite.
Therefore, $s$ is proper and $s^{-1}(C)$ must be finite. But
$s^{-1}(C)=\NN$, which is a contradiction.
\end{proof}

We obtain the following useful result.

\begin{Prop}
Let $X$ be an $S_2$-space and $C\subset X$. The following
statements are equivalent:
\begin{enumerate}
  \item[(i)] $C$ is s-compact.
  \item[(ii)] $C$ is countably compact.
  \item[(iii)] $C$ is sequentially compact.
\end{enumerate}
\end{Prop}

\begin{proof}[Proof]

By Theorem \ref{Got} (ii) and (iii) are equivalent; and by the
above lemma (i) implies (ii). Now we prove (iii) implies (i). Let
$C\subset X$ be a sequentially compact subset. Then $C$ is
sequentially closed since $X$ is sequentially Hausdorff. On the
other hand consider $s:\NN\to X$ any proper sequence. If
$s^{-1}(C)$ were infinite, then it would exist a subsequence
$s\circ u,$ $u\in\EE$ in $C$ and therefore a convergent
subsequence $s\circ u\circ v,$ $v\in\EE$, which is not possible by
Theorem \ref{sucprop}.
\end{proof}

\subsection{Brown's questions}

Given an open $U$ in a space $X$, we consider the set
$\Sigma_U(X)$ of sequences $s:\NN\to X$ such that $s\propto U$. It
is clear that  for any space $X$ with topology $\tau$, the family
$$\varepsilon=\{U\in\tau;\neg\Sigma_c(X)\subset\Sigma_U(X)\}$$
\noindent is a filter of open subsets. Now we consider the set
$X\cup\{\infty\}$ equipped with the topology
$\tau^{\wedge}=\tau\cup\varepsilon_\infty$, where
$\varepsilon_\infty=\{U\cup\{\infty\};U\in\varepsilon\}$. Thus we
get a space $X^{\wedge}$, which is the \textit{one-point
sequential compactification} defined by R. Brown \cite{Br}. If $X$
is already sequentially compact, then $\tau^{\wedge}$ is the
coproduct topology. Brown proves that: (1) $X^{\wedge }=X^+$ if
both $X$ and $X^+$ are $S_2$, and (2) $X^+$ is sequential if $X$
is first countable and a countable union of closed compact subsets
$K_i$ such that every compact subset is contained in some $K_i$.
Finally, he poses this problem: `find more general conditions for
$X^+$ to be sequential'.

Now we will give a satisfactory solution to this problem and
improve these earlier results. Namely, we will prove that the
statement `$X^+$ is sequential' is equivalent to `$X$ is
$\omega$-sequential'. First, we will see that, under the not very
restrictive condition of being $S_2$, the $\omega$-sequential
condition is equivalent to $X^{\wedge }=X^+$.

\begin{Th}\label{primcar}
An $S_2$-space $X$ is $\omega$-sequential if and only if
$X^\wedge=X^+$.
\end{Th}

\begin{proof}[Proof]

By Theorem \ref{sucprop}, $\Sigma_p(X)=\neg\Sigma_c(X)$. Hence the
s-compact subsets agree with the closed subsets $C$ such that
$\neg\Sigma_c(X)\subset\Sigma_U(X)$, where $U$ is the complement
of $C$ in $X$. When $X$ is $\omega$-sequential these subsets
clearly agree with the closed compact subsets of $X$, so
$X^{\wedge }=X^+$. Conversely, when $X^{\wedge }=X^+$ the
s-compact subsets agree with the closed compact subsets, that is,
$X$ is $\omega$-sequential.
\end{proof}

And now we give our characterization.

\begin{Th}\label{carac} A space $X$ is $\omega$-sequential if and
only if $X^+$ is sequential.
\end{Th}

\begin{proof}[Proof]

Suppose that $X$ is $\omega $-sequential and let $V\subset X^+$ be
any sequentially open subset. Then we must prove that $V$ is an
open subset of $X^+$. It is clear that $U=V-\{\infty\}\subset X$
is sequentially open in $X$ and therefore $U$ is an open subset in
$X$. If $\infty\notin V$ then $V=U$ is also open en $X^+$. Now
assume that $\infty\in V$; taking into account that a sequence in
$X$ is proper if and only if it converges to $\infty $ in $X^+$ we
have that $\Sigma_p(X)\subset\Sigma_U(X)$. Hence, the complement
$C$ of $U$ in $X$ is s-compact. By the hypothesis on $X$, $C$ is
closed compact, so $V\subset X^+$ is open.

Conversely, suppose that $X^+$ is a sequential space. Since $X$ is
open in $X^+$ then $X$ is also a sequential space. Now, we will
check that every s-compact subset $C$ is closed compact. If $U$ is
the complement of $C$ in $X$ and $V=U\cup\{\infty\}$, we must
prove that $V$ is open, that is, $V$ is a sequentially open subset
of $X^+$. Consider a sequence $s:\NN\to X^+$ such that $s\to x\in
V$; then we have $s\propto V$. Indeed, if $x\neq\infty$ or $s$ is
eventually constant at $\infty$, then the condition $s\propto V$
is clear. Otherwise, the complement $A$ of $s^{-1}(\infty)$ in
$\NN$ is infinite, and we can take the map $u\in\EE$ enumerating
$A$, so $t=s\circ u$ is a subsequence of $s$ contained in $X$.
Moreover $t$ is proper, so $t\propto U$ because $C$ is s-compact.
This fact implies that $s\propto V$.
\end{proof}

\begin{remark} As a consequence of the above theorem, if
$\mathbf{Pseq}_{lcH}$ denotes the full subcategory of
$\mathbf{P}_{lcH}$ whose objects are $\omega $-sequential spaces
and $\mathbf{Seq}^{\infty}_{cH}$ the full subcategory of
$\mathbf{Top}^{\infty }_{cH}$ whose objects are sequential spaces
we have, by Proposition \ref{equi}, an equivalence of categories
$$(-)^+:\mathbf{Pseq}_{lcH}\stackrel{\simeq
}{\longrightarrow } \mathbf{Seq}^{\infty}_{cH}.$$
\end{remark}

\subsection{Examples of $\omega $-sequential spaces.}

We finish this section giving some examples of $\omega$-sequential
spaces. For the topological notions and relations involved in the
next proposition see for instance \cite{Dug}.

\begin{Prop}
The following spaces are $\omega$-sequential:
\begin{enumerate}
\item[(i)] Sequential and paracompact spaces.

\item[(ii)] Sequential Lindel\"{o}f spaces.

\item[(iii)] Hausdorff and second countable spaces.
\end{enumerate}
\end{Prop}

\begin{proof}[Proof]

(i) Consider $C$ any s-compact subspace of $X$. Since $X$ is
sequential $C$ is closed, so it is also paracompact. On the other
hand $C$ is countably compact by Lemma \ref{lemon}. But every
paracompact countably compact space is compact.

(ii) Similarly, if $C$ is any s-compact subset, we have that $C$
is closed and therefore Lindel\"{o}f. Being $C$ countably compact
and Lindel\"{o}f, $C$ is compact.

(iii) Every second countable Hausdorff space is Lindel\"{o}f and
first countable (in particular sequential). Hence we may apply
(ii). \end{proof}

From the above proposition it follows that CW-complexes, metric
spaces, and usual topological manifolds are $\omega$-sequential
spaces. We can find other examples in the spaces $X$ considered by
R. Brown (first countable spaces which are countable union of
closed compact subsets $K_i$ such that every compact is contained
in some $K_i$). In this case $X^+$ is first countable and we may
apply Theorem \ref{carac}. The spaces studied by R. Brown are
$\sigma$-compact but CW-complexes are examples of
$\omega$-sequential spaces which are not necessarily
$\sigma$-compact.

The Hausdorff, locally compact and $\sigma$-compact spaces
considered by Edwards and Hastings in \cite{E-H} are also $\omega
$-sequential if they are also first countable.

Now we give an example of sequential space which is not
$\omega$-sequential. Consider $X=[0,\Omega)$ the ordinal space
\cite{Dug}, where $\Omega$ is the first non countable ordinal.
Then $X$ is first countable and Hausdorff and, in particular, it
is an $S_2$-space. On the other hand $X$ is sequentially compact
so $X^\wedge$ is the coproduct of $X$ with $\{\infty\}.$ Then
$X^\wedge\neq X^+$, since $X$ is not compact. By Theorem
\ref{primcar} we conclude that $X$ is not $\omega$-sequential.


\section{Sequentiality in exterior spaces}

As we have commented in the introduction, the category
$\mathbf{P}$ of spaces and proper maps does not have good
categorical properties so several classical sequential results
cannot be translated to $\mathbf{P}.$ Recall the embedding
$(-)^+:\mathbf{P}\hookrightarrow\mathbf{Top^{\infty}}$ given in
Proposition \ref{equi}. Now we will define a category equivalent
to $\mathbf{Top^{\infty}}$, which is the category {\bf E} of
exterior spaces and maps. This category, which was firstly
presented in \cite{G-G-H} and \cite{G-H}, contains the proper
category and is complete and cocomplete.  It is for this reason
that we complete our analysis considering the category of exterior
spaces.

\subsection{Exterior spaces}

Now we will provide some necessary background about exterior
spaces. For a detailed and ampler vision of this topic,
\cite{G-G-H} and \cite{G-H} can be consulted.

An \emph{externology} in a space $(X,\tau)$ is a nonempty
subfamily $\varepsilon\subset \tau$ which is a filter of open
subsets. This means that $\varepsilon $ is closed under finite
intersections and $U\in \tau,$ $E\in \varepsilon $, $U\supset E$
implies $U\in \varepsilon $. Note that $\varepsilon$ is a filter
in the lattice $\tau$, not a filter on the set $X$ as in set
theory. So in our algebraic sense, $\tau$ is the maximal filter of
open subsets, and $\{X\}$ is the minimal filter of open subsets
(see \cite{Jo}). An {\it exterior space}
$(X,\varepsilon\subset\tau )$ consists of a topological space
together with an externology. The elements of $\varepsilon$ are
called {\it exterior-open} subsets or, in short, {\it e-open}
subsets.

Given an space $X=(X,\tau)$, we may consider the \textit{discrete
exterior space} $X_d=(X,\varepsilon=\tau)$, and the
\textit{indiscrete exterior space} $X_i=(X,\{X\}\subset\tau)$. An
exterior space $(X,\varepsilon\subset\tau )$ is discrete if and
only the empty set is contained in the filter $\varepsilon $.

A function between exterior spaces,
$f:(X,\varepsilon\subset\tau)\to (X',\varepsilon'\subset\tau')$,
is said to be an {\it exterior map} if it is continuous and
$$f^{-1}(E)\in \varepsilon ,\;\forall E\in \varepsilon'.$$
The category of exterior spaces and exterior maps will be denoted
by {\bf E}.

The externology $\varepsilon _{cc}$ constituted by the complements
of the closed compact subsets (see Section 2) is called the {\it
cocompact externology}. The exterior space
$$X_{cc}=(X,\varepsilon _{cc}\subset\tau)$$ is a topological space
enriched with a system of open neighborhoods at infinity. It is
clear that we have a full and faithful functor
$$(-)_{cc}:{\bf P}\hookrightarrow {\bf E}$$ The category
{\bf E} has better properties than {\bf P}. In particular it has
limits and colimits given respectively by final and initial
structures with respect to the forgetful functor
$O:\mathbf{E}\to\mathbf{Top}$, which has two adjoints:
$(-)_d\,\dashv\,O\,\dashv\,(-)_i:\mathbf{Top}\to\mathbf{E}$.

\begin{remark}
We also note that $\emptyset$ ($\emptyset_d=\emptyset_i$) is
the initial object in {\bf E}, and $\mathbf{1}_i$ is the final
object. It is clear that $\mathbf{1}_i$ and $\mathbf{1}_d$ are
non-isomorphic exterior spaces. Every exterior map
$\mathbf{1}_d\to X$ can be considered as an element of $X$, but an
exterior map $\mathbf{1}_i\to X$ is an element which belongs to
all the open subsets of the filter.

Let us call \emph{limit} of $(X,\varepsilon\subset\tau)$ the
intersection $\ell(X)$ of all open subsets of the filter. Elements
of $\ell(X)$ are \emph{limit points} of the exterior space. For
every subset $A\subset X$, the space $X$ endowed with the filter
$\mathcal{U}(A)=\{U\in\tau;A\subset U\}$ is an exterior space such
that $\ell(X)=A$. It is clear that $\ell(X_d)=\emptyset$, but also
$\ell(\NN)=\emptyset$, where $\mathbb{N}$ is the discrete space
with the cofinite externology $\varepsilon_{cc}$. For every
exterior map $f:X\to Y$ we have $f(\ell(X))\subset\ell(Y)$. In
particular, every exterior map $X_i\to Y$ must have its image
contained in $\ell(Y)$. Thus
$\;(-)_i\,\dashv\,\ell:\mathbf{E}\to\mathbf{Top}$, where we
consider $\ell(Y)$ as an indiscrete space.
\end{remark}

If $X$ is an exterior space, an \emph{exterior sequence} is an
exterior map $\mathbb{N}\rightarrow X$, where $\mathbb{N}$ is the
discrete space with the externology $\varepsilon _{cc}$, that is,
the cofinite externology. In other words, a sequence
$s:\mathbb{N}\rightarrow X$ is exterior if
$$s\propto U,\;\; \forall U\in\varepsilon.$$ The set of all
exterior sequences $\mathbb{N}\rightarrow X$ will be denoted by
$\Sigma_e(X)$. In particular, $\EE$ is a submonoid of
$\Sigma_e(\NN)=\Sigma_p(\NN)$, so $\Sigma_e(X)$ is an $\EE$-subset
of $X^\NN$. We also note that Theorem \ref{sucprop} is a very
particular case, because there are exterior spaces satisfying
$\Sigma_e(X)\subset\Sigma_c(X)$. Take for instance the filter of
open subsets $\mathcal{U}(\{x\}),\;x\in X$.

\subsection{Sequentially exterior maps, and e-sequential exterior spaces}

Now we analyze the sequentiality of exterior spaces and maps,
extending the sequentiality of spaces and proper maps. The next
definition and theorem are analogous to Definition \ref{seprop}
and Theorem \ref{equiva-prop} respectively. Now we formulate them
in the general exterior scope.

\begin{Df}
A map between exterior spaces $f:X\to Y$ is said to be
\emph{sequentially exterior} or \emph{e-sequential} if it is
sequentially continuous and it preserves exterior sequences, that
is,
\begin{enumerate}
\item
$f\circ s\in \Sigma_c^+(Y),$ for all $s\in \Sigma_c^+(X);$ and

\item
$f\circ s\in \Sigma_e(Y),$ for all $s\in \Sigma_e(X).$
\end{enumerate}
In other words $f\circ\Sigma_c^+(X)\subset\Sigma_c^+(Y)
\;\;\text{and}\;\; f\circ\Sigma_e(X)\subset\Sigma_e(Y).$
\end{Df}

If $(X,\varepsilon _{X}\subset \tau _{X})$ is an exterior space
and $\infty$ is a point which does not belong to $X$ then we
consider the based space $X^\infty=X\cup\{\infty\}$ with base
point $\infty$, and topology $\tau^\infty=\tau _{X}\cup
\{E\cup\{\infty\}:E\in \varepsilon _{X}\}$. In this way the
canonical inclusion $X\hookrightarrow X^\infty$ is a homeomorphism
onto its image. Notice that Brown's construction $X^\wedge$ in
Section 3 is a particular case of $X^\infty$.

If $f:X\to X'$ is an exterior map, we define $f^{\infty}:X^\infty
\rightarrow {X'}^\infty$ by $f^{\infty}(x)=f(x)$ when $x\in X$ and
$f^{\infty }(\infty )=\infty$. Thus we obtain a functor
$$(-)^{\infty }:{\bf E}\rightarrow {\bf Top}^\infty,$$ which is an
equivalence of categories. Indeed, an easy verification shows that
it is full and faithful. On the other hand, the quasi-inverse is
described as follows: Let $(X,x_{0})$ be an object in ${\bf
Top}^\infty;$ then consider $\bar{X}=X\mbox{-}\{x_{0}\}$ equipped
with
$$\tau_{\bar{X}}=\{A\mbox{-}\{x_{0}\}:A\in \tau _{X}\}\subset \tau
_X,\;\;\varepsilon_{\bar{X}}=\{A\mbox{-}\{x_{0}\}:A\in \tau
_{X},x_{0}\in A\}.$$ Then $(\bar{X}^\infty,\infty )\cong
(X,x_{0})$ in ${\bf Top}^\infty$ by means of
$$\alpha:(\bar{X}^{\infty },\infty )\rightarrow (X,x_{0}),\;\;\alpha (x)=\left\{
\begin{array}{ll} x, & \mbox{if}\hspace{5pt}x\in \bar{X} \\
{\infty ,} & \mbox{if}\hspace{5pt}x=x_0
\end{array}\right.$$
As a consequence, the category ${\bf Top^{\infty }}$ is complete
and cocomplete. The initial object is $(\mathbf{1},0)$ and the
final object is $(\mathbf{2}_S,0)$. It is also clear that we have
a commutative diagram
$$\xymatrix{
  {\mathbf{P}} \ar[dr]_{(-)^+} \ar[r]^{(-)_{cc}}
                & {\mathbf{E}} \ar[d]^{(-)^\infty}_{\simeq }  \\
                & {\mathbf{Top}^\infty}  }$$

 The next result gives us a relationship between the
sequentially exterior maps and the sequential continuous maps when
we consider the functor $(-)^\infty$. We leave the proof to the
reader, since it is analogous to that given in the previous
section.

\begin{Th}
Let $f:X\to Y$ be a map between exterior spaces. Then $f$ is
e-sequential if and only if $f^\infty:X^\infty\to Y^\infty$ is
sequentially continuous.
\end{Th}

As in the proper case we want to find a suitable class of exterior
spaces in which exterior maps and sequentially exterior maps
agree. Now we will give a notion which corresponds in the proper
case to the complements of s-compacts. In fact, Definition
\ref{omega} and Proposition \ref{omega2} have a natural extension
to the exterior setting.

\begin{Df}\label{eseq}
Let $X$ be an exterior space.
\begin{enumerate}
\item[(i)] We say that $E\subset X$ is a \emph{sequentially e-open}
 subset if it is a sequentially open subset and every exterior
sequence is eventually in $E$.

\item[(ii)] $X$ is said to be \emph{e-sequential} if it is a
sequential space and every sequentially e-open subset is e-open.
\end{enumerate}
\end{Df}

Given a space $X$, we have the family of subsets
$$\varepsilon_{seq}(X)=\{U\subset X;U\; \text{is sequentially e-open}\}.$$
Obviously, when $X$ is an e-sequential space then
$\varepsilon_{seq}(X)$ is a filter of open subsets.

\begin{remark}
Let us recall that an \emph{exterior set} (see \cite{G-H-R})
is an exterior space with discrete topology. The category
$\mathbf{ESet}$ of exterior sets is a full subcategory of
$\mathbf{E}$. Given an exterior set $(X,\varepsilon)$, the final
filter on $X$ with respect to $\Sigma_e(X)$ is
$$\bar{\varepsilon}=\{U\subset X;s\propto U,\forall
s\in\Sigma_e(X)\}\;.$$ It is clear that
$\varepsilon\subset\bar{\varepsilon}$, and, following Definition
\ref{eseq}, an exterior set is e-sequential if
$\varepsilon=\bar{\varepsilon}$. This particular definition was
used in \cite{E-G-M}.
\end{remark}

We obtain the following immediate result:

\begin{Prop}\label{caracext}
Let $f:X\rightarrow Y$ be a map between exterior spaces. Then
\begin{enumerate}
\item[(i)] $f:X\rightarrow Y$ is e-sequential if and only if it is
sequentially continuous and $f^{-1}(V)\in\varepsilon_{seq}(X),
\;\forall V\in\varepsilon_{seq}(Y).$

\item[(ii)] Suppose that $X$ is e-sequential. Then $f$ is exterior if
and only if $f$ is e-sequential.
\end{enumerate}
\end{Prop}

The next result is completely analogous to that of Theorem
\ref{carac} so its proof is omitted and left to the reader.

\begin{Th}\label{infty}
An exterior space $X$ is e-sequential if and only if $X^{\infty }$
is a sequential space.
\end{Th}

\begin{remark}If $\mathbf{Seq}^{\infty }$ denotes the full subcategory of
$\mathbf{Top}^{\infty }$ whose objects are sequential spaces, and
$\mathbf{Eseq}$ the full subcategory of $\mathbf{E}$ whose objects
are the e-sequential exterior spaces, then we have an equivalence
of categories $$(-)^{\infty}:\mathbf{Eseq}\stackrel{\simeq
}{\longrightarrow}\mathbf{Seq}^{\infty}.$$
\end{remark}

\medskip

As a new example of exterior spaces we consider the exterior
version of first countable spaces, which are sequential. After a
convenient definition, properties of exterior spaces become
properties of spaces.

\begin{Df} Let $X$ be an exterior space.
\begin{enumerate}
\item[(i)] An \textrm{exterior base} in $X$ is a nonempty
collection of e-open subsets $\beta\subset\varepsilon$ such that
for every e-open $E$ one can find $B\in\beta$ such that $B\subset
E$.

\item[(ii)] $X$ is said to be \textrm{e-first countable}
(or \textrm{first countable at infinity}) if it is first countable
and it has a countable exterior base.
\end{enumerate}
\end{Df}

\begin{Prop}\label{carinfty}${}$

\begin{enumerate}
\item[(i)] An exterior space $X$ is e-first countable if and only if
$X^{\infty }$ is a first countable space.

\item[(ii)] Every e-first countable exterior space $X$ is
e-sequential.
\end{enumerate}
\end{Prop}

\begin{remark}
Proposition \ref{carinfty}(ii) follows immediately from Theorem
\ref{infty} and Proposition \ref{carinfty}(i), but it is also easy
to prove it directly. In general, these exterior spaces are
simpler to handle. They are the equivalent, in the category of
pro-spaces, to the towers of spaces (see for example \cite{E-H}).
In addition, they play an important role in the theory of
sequential homology, which is defined in \cite{G-H}. As examples
of this nature we can mention the ones of the form $X_{cc},$ where
$X$ is a first countable, $\sigma$-compact, locally compact
Hausdorff space.

On the other hand if $\mathbf{E}_{efc}$ denotes the full
subcategory of $\mathbf{E}$ whose objects are e-first countable
exterior spaces, and $\mathbf{Top}^{\infty}_{fc}$ denotes the full
subcategory of $\mathbf{Top}^{\infty}$ whose objects are first
countable spaces, then we have an equivalence
$$(-)^{\infty}:\mathbf{E}_{efc}\stackrel{\simeq }{\longrightarrow }
\mathbf{Top}^{\infty}_{fc}.$$
\end{remark}

Naturally, $F\subset X$ is a {\it sequentially e-closed} subset if
its complement in $X$ is sequentially e-open. Every e-open
(e-closed) subset is sequentially e-open (e-closed) subset. Recall
that a sequential space $X$ has the final topology with respect to
its continuous maps $\mathbb{N}^+\to X$. In the same way, it is
clear that an e-sequential space $X$ has the final externology
with respect to its continuous maps $\mathbb{N}^+\to X$ and its
exterior sequences $\mathbb{N}\to X$.

Many properties of the e-sequential exterior spaces are analogous
to the properties of sequential spaces, and they are proved in a
similar way. For instance, $\mathbf{Eseq}$ is a coreflective
subcategory of $\mathbf{E}$. Recall that given an exterior space
$X=(X,\varepsilon\subset\tau)$, we can consider $\tau_{seq}$  and
$\varepsilon_{seq}=\{U\subset X;U\;\;\mbox{is sequentially
e-open}\}\subset \tau_{seq}.$ Then we have an e-sequential
exterior space $\sigma X=(X,\varepsilon_{seq}\subset\tau_{seq})$.
With $\sigma(f)=f$ we complete the definition of the coreflector
functor
$$\sigma:\mathbf{E}\to \mathbf{Eseq}.$$ The subcategory
$\mathbf{Eseq}$ has limits and colimits. The colimits are taken in
$\mathbf{E}$ and for limits in $\mathbf{Eseq}$ we take limits in
$\mathbf{E}$ and then we apply the coreflector functor. The counit
of the coreflection is the identity $\sigma X\to X$, and an
exterior space is e-sequential if and only if the counit is an
isomorphism. Hence:

\begin{Prop}Let $X$ be an exterior space. The following statements
are equivalent:
\begin{enumerate}
\item[(i)] $X$ is e-sequential.

\item[(ii)] For all map $f:X\to Y$ with $Y$ exterior space, $f$ is
exterior if and only if $f$ is e-sequential.
\end{enumerate}
\end{Prop}

\begin{remark}
The property of being e-sequential is not hereditary since
being sequential is not. Nevertheless, it is straightforward to
check that in an e-sequential exterior space the e-open and
e-closed subsets are e-sequential exterior spaces.

On the other hand the image of an e-sequential exterior space is
not necessarily e-sequential. Take for instance an identity map
$id_X:(X,\{X\}\subset\tau_d)\to (X,\{X\}\subset\tau )$ where
$\tau_d$ is the discrete topology and $\tau$ is any non-sequential
topology. However, an exterior quotient (that is, a quotient space
with the corresponding final externology) of any e-sequential
exterior space is an e-sequential exterior space.
\end{remark}

If $X$ is any space, an immediate verification shows that $X$ is
$\omega$-sequential if and only if $X_{cc}$ is e-sequential. Hence
we have a commutative diagram

$$\xymatrix{
 {\mathbf{Eseq}}\; \ar@{^{(}->}[r]   & {\mathbf{E}}   \\
 {\mathbf{Pseq}}\; \ar@{^{(}->}[r] \ar@{^{(}->}[u]^{(-)_{cc}}
 & {\mathbf{P}} \ar@{^{(}->}[u]_{(-)_{cc}}
 }$$

\noindent where $\mathbf{Pseq}$ denotes the full subcategory of
$\mathbf{P}$ whose objects are $\omega$-sequential spaces. Of
course, there are e-sequential exterior spaces which do not come
from any $\omega$-sequential space; for instance, a first
countable space $X$, provided with the indiscrete externology.

Finally, we give an example of space $X$ such that $\sigma
(X_{cc})$ does not belongs to $\mathbf{Pseq}$. Take the exterior
space $X_{cc}$, where $X=[0,\Omega)$ is the ordinal space
considered at the end of Section 3. It is clear that
$\sigma(X_{cc})$ is equal to $X$ as topological spaces, because
$X$ is sequential. On the other hand $\Sigma_p(X)=\emptyset$, so
$\sigma(X_{cc})$ has the discrete externology. Hence
$\sigma(X_{cc})$ is not $X_{cc}$ because $X$ is not compact.

\subsection{Sequential exterior spaces as sheaves}

We finish this paper showing a new extension of results from
topological spaces to exterior spaces. Namely we will prove that
the category $\mathbf{Eseq}$ is a full subcategory of certain
topos of sheaves. Subsequently, since
$(-)_{cc}:\mathbf{Pseq}\hookrightarrow \mathbf{Eseq}$ is a full
embedding, $\mathbf{Pseq}$ will also be a full subcategory of a
topos of sheaves.

Let $\CC$ be the subcategory of $\mathbf{E}$ defined by the
objects $1\equiv 1_d,$ $\NN^+\equiv \NN^+_d,$ $\NN\equiv \NN
_{cc}$ and the following morphisms:
$$
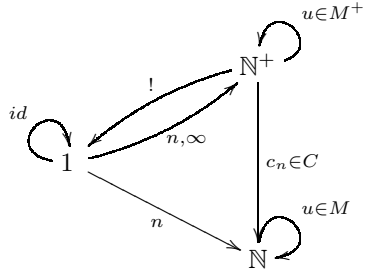
\xymatrix
         {
            & &  {\NN^+}  \ar@(r,u)[]_{u\in M^+}
              \ar@/_/[lld]_{!}  \ar[dd]^{c_n\in C}  \\
 1 \ar@(l,u)[]^{id}  \ar[rrd]_{n}    \;\ar@/_/[rru]_{n,\infty}        &  &  \\
                            & &  {\NN}  \ar@(u,r)[]^{u\in M}
         }
$$

\noindent where $C$ denotes the set of all constant maps
$c_n=n\circ !:\NN^+\to\NN$ (we shall also use $C$ for the constant
maps $\NN^+\rightarrow \NN^+$), and the monoids on the objects
$\NN^+, \NN$ are respectively $M^+=\Sigma_c^+(\NN^+)$ and
$M=\Sigma_e(\NN)$.

Let us denote $\mathbf{\CC Set}$ the category of presheaves over
$\CC$ or $\CC$-sets. Then, we may see each $\CC$-set as a
commutative diagram
$$
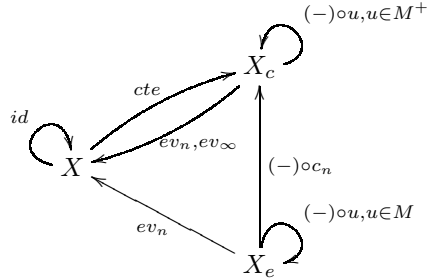
\xymatrix
         {
    & &  X_c  \ar@/^/[lld]^{ev_n,ev_\infty}
    \ar@(r,u)[]_{(-)\circ u,u\in M^+}   \\
 X \ar@(l,u)[]^{id}      \ar@/^/[rru]^{cte}        &  &  \\
    & &  X_e \ar[llu]^{ev_n}  \ar@(u,r)[]^{(-)\circ u,u\in
    M}  \ar[uu]_{(-)\circ c_n}
         }$$

\noindent where $X_c$ is an $M^+$-set, $X_e$ an $M$-set, and the
compositions with the formal constant and evaluation maps are
clear: $c_n=cte\circ ev_n$, $ev_n(s\circ u)=ev_{u(n)}(s)$,
etcetera.  Briefly, we shall denote a $\CC$-set $P$ as a triple
$$P=(X,X_c,X_e)$$ \noindent where the morphisms can be forgotten
in the notation. Given another $\CC$-set, $Q=(Y,Y_c,Y_e)$, a
$\CC$-map $\phi:P\to Q$ is given by a triple $\phi=(f,f_c,f_e)$,
where $f:X\to Y$ is a map, $f_c:X_c\to Y_c$ an $M^+$-map
(equivariant), $f_e:X_e\to Y_e$ an $M$-map, and the compositions
with constant and evaluation maps are satisfied.

\medskip

It is clear that a functor $\Sigma:\mathbf{E}\to\mathbf{\CC Set}$
is defined by $\Sigma(X)=\mathbf{E}(-,X).$ Briefly,
$$\Sigma(X)=(X,\Sigma_c^+(X),\Sigma_e(X)),$$ and if $f:X\to Y$ is
an exterior map, then the natural transformation
$\Sigma(f):\Sigma(X)\to\Sigma(Y)$ is formed by $f$,
$f_c=f\circ(-):\Sigma_c^+(X)\to\Sigma_c^+(Y)$, and
$f_e=f\circ(-):\Sigma_e(X)\to\Sigma_e(Y)$. Note that the functor
$\Sigma_c^+(X)$ was used in \cite{J} to embed $\mathbf{Seq}$ in
the topos $\mathbf{M^+Set}$ of  $M^+$-sets, and
$\Sigma_e:\mathbf{E}\to \mathbf{MSet}$ is a construction used in
\cite{G-H-R}.

If we consider the Yoneda embedding $y:\CC\hookrightarrow
\mathbf{\CC Set}$ and the representable $\CC$-sets we get
$y(1)=\Sigma(1_d)=(1,1,\emptyset)$,
$y(\NN^+)=\Sigma(\NN^+_d)=(\NN^+,M^+,\emptyset)$, and
$y(\NN)=\Sigma(\NN)=(\NN,C,M)$. That is, the Yoneda functor is the
restriction of $\Sigma$ to $\CC$. For instance, the diagram of
$y(\NN)$ is, with $\NN\cong C$:
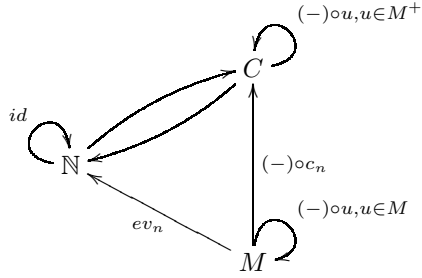
$$
\xymatrix{
    & &  C  \ar@/^/[lld]  \ar@(r,u)[]_{(-)\circ u,u\in M^+}   \\
 {\NN} \ar@(l,u)[]^{id}      \ar@/^/[rru]        &  &  \\
    & & M \ar[llu]^{ev_n}  \ar@(u,r)[]^{(-)\circ u,u\in M}  \ar[uu]_{(-)\circ c_n}
         }$$

We point out that since $\mathbf{E}$ has colimits and
$\Sigma(X)(A)=\mathbf{E}(A,X)$ for any object $A$ of $\CC$, then
$\Sigma$ has a left adjoint by the universal property of the
categories of presheaves. Moreover all the three objects of $\CC$
are e-sequential, hence the last remark is also true when $\Sigma$
is defined on $\mathbf{Eseq}$.

\medskip

Actually, the functor $\Sigma$ has a better behavior when its
domain is restricted to $\mathbf{Eseq}$. In fact, we have the
following result:

\begin{Prop}\label{embed1}
The functor $\Sigma:\mathbf{Eseq}\to\mathbf{\CC Set}$ is full,
faithful and injective on objects.
\end{Prop}

\begin{proof}[Proof]
It is clear that $\Sigma$ is faithful, and it is injective on
objects because $\mathbf{Eseq}$ is a coreflective subcategory of
$\mathbf{E}$. Finally, given a natural transformation
$(f,f_c,f_e):\Sigma(X)\to\Sigma(Y)$, it is easy to verify by
naturality that $f_c=f\circ(-)$; and the same is true in the
exterior component $f_e$. Hence a categorical reading of
Proposition \ref{caracext}(ii) shows that $\Sigma$ is full.
\end{proof}

From now on we consider e-sequential spaces as $\CC$-sets by the
full embedding $\Sigma:\mathbf{Eseq}\hookrightarrow\mathbf{\CC
Set}$. Thus an e-sequential space is completely determined by its
points, convergent sequences an exterior sequences, together with
the maps between them given by the notion of $\CC$-set. Our last
goal is to determine a smaller topos of sheaves
$\mathcal{E}\hookrightarrow\mathbf{\CC Set}$ such that
$\Sigma:\mathbf{Eseq}\hookrightarrow\mathcal{E}$.

\medskip

Recall from \cite{J} (with a different formulation based on
monoids) that given a topological space $X$, the $M^+$-set
$\Sigma_c^+(X)$ is a sheaf when we consider on the monoid $M^+$
the Grothendieck topology $\JJ_c$ formed by the ideals I such
that:
\begin{enumerate}
\item[(i)] $C\subset I$,

\item[(ii)] $\forall u\in E,\;\exists v\in E\; ; u\circ
v\in I$,
\end{enumerate}
\noindent where $C$ is the set of all constant maps in $M^+$, and
$E$ is the submonoid of $M^+$ constituted by all the monotone and
injective maps $u:\NN^+\to\NN^+$ (hence
$u^{-1}(\infty)=\{\infty\}$). We can translate this scheme to the
monoid $M$, which has not constants. By deleting
$u(\infty)=\infty$, we can see $E$ as a submonoid of $M$ and
define a family  $\JJ_e$ of ideals of $M$ by means of the
condition (ii) above. Then $\JJ_e$ is a Grothendieck topology on
$M$ \cite{E-G-M} (actually the double negation topology). The
proof of the following result is long but straightforward.

\begin{Le}
For any object $A$ of $\CC$ we define a collection $\JJ(A)$ of
families of morphisms of $\CC$ with codomain $A$ as follows:
\begin{enumerate}
\item[(i)] The unique family in $\JJ(1)$ is all the morphisms with
codomain $1$.

\item[(ii)] A family in $\JJ(\NN^+)$, denoted $\hat I$, is formed by
all the morphisms $1\to\NN^+$ and the morphisms $\NN^+\to\NN^+$ of
an ideal $I\in\JJ_c$.

\item[(iii)] A family in $\JJ(\NN)$, denoted $\hat I$, is formed by
all the morphisms $1\to\NN$,  all the morphisms $\NN^+\to\NN$, and
the morphisms $\NN\to\NN$ of an ideal $I\in\JJ_e$.
\end{enumerate}
Then $\JJ$ is a Grothendieck topology on $\CC$.
\end{Le}

We denote the subtopos of sheaves
$\mathcal{E}=sh(\CC,\JJ)\hookrightarrow\CC Set$.

\begin{Th}
The embedding $\Sigma:\mathbf{Eseq}\hookrightarrow\mathcal{E}$
holds.
\end{Th}
\begin{proof}[Proof]
By Proposition \ref{embed1}, it suffices to prove that $\Sigma(X)$
is a sheaf for any e-sequential space $X$. This should be done on
the three objects $1,\NN^+,\NN$. The case 1 is obvious and for the
case $\NN^+$ we refer the reader to \cite{J}. The case $\NN$ is
analogous to the last one. We consider $\hat I$ as a subpresheaf
of $y(\NN)$. Given a natural transformation $\theta:\hat
I\to\Sigma(X)$ we must find a unique $s\in\Sigma_e(X)$ such that
$\theta$ is the restriction of $y(s)=(s,s\circ(-),s\circ(-))$. But
$\theta$ is of the form $\theta=(s,s\circ(-),H)$, where $s:\NN\to
X$ and $H:I\to\Sigma_e(X)$ is an M-equivariant map; and for any
$n\in\NN$, $ev_n\circ H=s\circ ev_n$, that is, $H(g)=s\circ g$ for
any $g\in I$. Since $I\in\JJ_e$ and each $H(g)$ is exterior
\cite{E-G-M} the latter condition means that $s$ is exterior.
\end{proof}

\begin{remark} Note that the sequentiality is not needed for the sheaf
condition, but only to apply Proposition \ref{embed1}.
\end{remark}

\end{document}